\documentclass[letterpaper,11pt]{article}

\usepackage{times}
\usepackage[margin=1in]{geometry}
\usepackage{latexsym, amsthm, amscd, amsfonts, mathrsfs, amsmath, amssymb, stmaryrd, tikz-cd, mathrsfs, bbm, url, esint, mathtools, enumerate}

\newtheorem{defin}{Definition}
\newtheorem{thm}{Theorem}
\newtheorem{prop}{Proposition}
\newtheorem{lemma}{Lemma}
\newtheorem{corol}{Corollary}
\newtheorem{conj}{Conjecture}
\DeclareMathOperator\BSC{\mathrm{BSC}}
\DeclareMathOperator\BEC{\mathrm{BEC}}
\DeclareMathOperator\arctanh{arctanh}
\DeclareMathOperator\Unif{\mathrm{Unif}}
\newcommand{\bP}{\mathbb{P}}
\newcommand{\bE}{\mathbb{E}}

\newcommand{\mR}{\mathbb{R}}

\DeclareMathOperator{\Var}{Var}

\newcommand{\Y}{\mathcal{Y}}
\newcommand{\Z}{\mathcal{Z}}

\newcommand{\Po}{\mathrm{Poi}}

\begin{document}

\title{Stochastic block model entropy and broadcasting on trees with survey}

\author{Emmanuel Abbe \thanks{Institute of Mathematics, EPFL, Lausanne, CH-1015, Switzerland. Email: emmanuel.abbe@epfl.ch.} \and
Elisabetta Cornacchia \thanks{Institute of Mathematics, EPFL, Lausanne, CH-1015, Switzerland. Email: elisabetta.cornacchia@epfl.ch.} \and
Yuzhou Gu \thanks{IDSS, LIDS, and Dept. of EECS, MIT, Cambridge, MA, 02139, USA. Email: yuzhougu@mit.edu.} \and
Yury Polyanskiy \thanks{IDSS, LIDS, and Dept. of EECS, MIT, Cambridge, MA, 02139, USA. Email: yp@mit.edu.}}
\date{}

\maketitle


\begin{abstract}
The limit of the entropy in the stochastic block model (SBM) has been characterized in the sparse regime for the special case of disassortative communities  \cite{10.1145/3055399.3055420} and for the classical case of assortative communities but in the dense regime \cite{DAM15}. The problem has not been closed in the classical  sparse and assortative case. This paper establishes the result in this case for any SNR besides for the interval $(1,3.513)$. It further gives an approximation to the limit in this window.

The result is obtained by expressing the global SBM entropy as an integral of local tree entropies in a broadcasting on tree model with erasure side-information. The main technical advancement then relies on showing the irrelevance of the boundary in such a model, also studied with variants in~\cite{kanade2014global}, \cite{Mossel_2016} and~\cite{mossel2015local}. In particular, we establish the uniqueness of the BP fixed point in the survey model for any SNR above 3.513 or below 1. This only leaves a narrow region in the plane between SNR and survey strength where the uniqueness of BP  conjectured in these papers remains unproved. 
\end{abstract}



\section{Introduction}

Over the last decade, several works have established a precise picture for the statistical and algorithmic behavior of the stochastic block model (see an account in \cite{abbe2017community}). In particular, the questions of weak and exact recovery, i.e., whether it is possible (or not) to recover the communities in the extremal cases of weak and exact accuracy, have been fully closed in the two-community symmetric SBM by establishing sharp threshold phenomena in terms of appropriate signal-to-noise (SNR) ratios \cite{10.1145/2591796.2591857, MNSplanted15,MNS14,abh}. 
Yet, despite significant  progress, the more nuanced question of proving how much information or agreement can be recovered about the communities at any given value of the SNR has remained open even in this simplest case.

More specifically, for two symmetric communities and in the sparse regime, the expression of the limiting entropy of the
SBM is characterized\footnote{Characterizing the limit does not mean obtaining an explicit expression; it refers to an
implicit $n$-independent expression relying on integrals and fixed point equations for the quantities of interest in all
of the papers discussed here.} at all SNR for the special case of disassortative communities (i.e., communities that
connect more outside than inside) \cite{10.1145/3055399.3055420}. The problem for assortative communities is closed but
in the denser regimes, where the vertex degrees diverge while maintaining a finite SNR \cite{DAM15}. However, for the
classical case of assortative communities and in the sparse regime, a complete characterization remains open, despite
significant progress \cite{kanade2014global,Mossel_2016,mossel2015local}. The expression of the optimal agreement
(rather than the entropy) is known in this case for SNR large enough, and is related to the
problem of robust reconstruction on a tree \cite{Mossel_2016}. This result is conjectured to hold all the way down to
the optimal threshold of 1, i.e., the threshold until which the communities can be weakly recovered. We make progress on
this question by establishing the result down to 3.513. Further, we establish new results and improvements of prior
results for the problem of broadcasting on a tree with side-information; see Section \ref{sec:related}.

\paragraph{The SBM entropy.} Recall that in the symmetric SBM with two communities, a random variable $X$ is drawn uniformly at random in $\{\pm 1\}^n$ and an $n$-vertex graph $G$ is drawn by connecting vertices having same (resp.\ different) values in $X$ with probability $a/n$ (resp.\ $b/n$).

The SBM mutual information is defined by the limit (if it exists)
\begin{align}
\mathcal{I}(a,b) :=\lim_{n \to \infty} \frac{1}{n} I(X;G),
\end{align}
where $I$ is the mutual information.
Note that establishing the existence of this limit is nontrivial. This was proved in \cite{abbetoc} for the case of $a<b$, the same case for which the value of the limit has more recently been established \cite{10.1145/3055399.3055420}.
Note also that due to the chain rule $I(X;G)=H(X)- H(X|G)$, the SBM mutual information is the complement of the SBM conditional entropy (called simply the SBM entropy)
\begin{align}
\mathcal{H}(a,b) := \lim_{n \to \infty} \frac{1}{n} H(X|G).
\end{align}
Informally, the SBM mutual information measures how much information can be recovered about the communities after observing the graph, and equivalently, the SBM entropy measures how much uncertainty is left about the communities after observing the graph. More formally, it quantifies the average number of bits needed to represent the communities after observing the graph; see \cite{7852203} for formal relations to graph compression. 

Note that one may use other measures on the communities signal given the graph, such as the optimal (normalized) mean square error of reconstructing the $n \times n$ rank-2 block matrix (with $a/n$ in the $n/2 \times n/2$ diagonal blocks and $b/n$ in the off diagonal blocks), or the optimal (normalized) agreement (Hamming distance) of reconstructing $X$ up to a community relabelling. These can be explicitly related to each other in the tree models discussed next, and require bounds in the SBM context; see for instance \cite{DAM15}. The conditional entropy allows however for a direct reduction from the SBM to the tree model with side-information, as discussed below.

\paragraph{The BOTS entropy.} Consider the following problem of broadcasting on a tree with side-information (BOTS).
This will be later defined on general trees and with general side-information, but consider for simplicity the case of regular
trees (where each vertex has exactly $d$ descendants) and erasure side-information. In this model, a random bit is
attached to the root of the tree and broadcasted down the tree by flipping its value independently with probability
$\delta$ on each edge (for convenience we call $\theta = 1-2\delta$). We denote by $\sigma_\rho$ the root-bit, by $\sigma_{L_k}$ the $d^k$-dimensional vector of the
leaf-bits at generation $k$, and by $\omega_{T_k}^\epsilon$ the side-information up to depth $k$: these are the vertices
labelled that are revealed in the tree (besides the root) independently with probability $1-\epsilon$. We call this side-information the ``survey''. Note that this is the type of side-information used in our connection between BOTS and SBM entropies, but other types of side-information are of independent interest. In this paper we devote attention to general (but symmetric with respect to the spin flip) observation model of the nodes, which we refer to as the BMS channel $W$, see Appendix~\ref{apx:bms} for the explanation of this notion.

We are now interested in two quantities:
\begin{enumerate}
\item the limiting entropy of the root-bit after observing the leaf-bits and the survey, i.e.\footnote{Note that in these tree models, the limits can be proved to always exist.}, $$\bar{h}(d,\theta,\epsilon):=\lim_{k \to \infty} H(\sigma_\rho|\sigma_{L_k}, \omega_{T_k}^\epsilon),$$
\item the same quantity without the leaf bits being observed, i.e., $$h(d,\theta,\epsilon):=\lim_{k \to \infty} H(\sigma_\rho |\omega_{T_k}^\epsilon).$$
\end{enumerate}

We now give a rather direct method to express the SBM entropy in terms of BOTS entropies. 

\paragraph{SBM to BOTS entropy reduction.}
The relation obtained between the SBM and BOTS conditional entropy is as follows: if for some range of parameters $d,\theta$, we can establish that $$h=\bar{h}, \quad \forall \epsilon \in (0,1),$$
i.e., if the {\it boundary is irrelevant}, 
then we can characterize $\mathcal{H}$ as an integral of $\bar h$ using the parameter correspondence $d=(a+b)/2$ and $\theta = \frac{a-b}{a+b}$ (see Theorem \ref{thm:HSBM-BOT}). 

Our starting point to such a reduction is an area-theorem or interpolation trick that is commonly used in coding theory \cite{RU01} and related statistical physics literature \cite{MMBook}.

The idea is to express the entropy in the SBM $H(X|G)$ as the integral 
\begin{align}
    \frac{1}{n} H(X|G)= \int_0^1  \frac{1}{n} \frac{\partial}{\partial \epsilon} H(X|G, Y^\epsilon) d \epsilon,
\end{align}
where, similarly as before, $Y^\epsilon$ is an erasure survey that reveals the community of each vertex in $X$ independently with probability $1-\epsilon$. 
We then use the fact that $\frac{1}{n}\frac{\partial}{\partial \epsilon} H(X|G, Y_\epsilon) = H(X_1 | G, Y^\epsilon_{\sim 1} )$,
where $1$ is an arbitrary vertex in the graph and $Y^\epsilon_{\sim 1}$ denotes the erasure survey on all vertices excluding vertex $1$. Since conditioning reduces entropy, one can  upper-bound $H(X_1 | G, Y^\epsilon_{\sim 1} )$ by considering only the information in the vertex $1$ neighborhood, 
and due to the local tree-like topology of SBMs, this gives an upper-bound with the BOTS entropy without leaf-information. Moreover, one can add the leaf-information in the conditioning to cut-off the  graph beyond a local neighborhood, using the Markovianity\footnote{Strict Markovanity does not hold in the SBM due to the weak effect of non-edges, and this requires a technical lemma; see proof of Theorem \ref{thm:HSBM-BOT}. This technicality can also be avoided by considering the related Censored Block Model (CBM), rather than the SBM, for which strict Markovianity holds.} of the model,  obtaining as well a lower-bound from the BOTS entropy but this time with the leaf-information, cf.~\eqref{eq:boundsentropy}.    

Different kind of reductions from SBMs to tree models have long been known and leveraged in the SBM in
\cite{10.1145/3055399.3055420,Mossel_2016,alaoui2019computational}; we refer to Section \ref{sec:related} for further discussions on these. 

We now turn to the crux of the analysis, i.e., the establishment of $h=\bar{h}$. 

\paragraph{Uniqueness of BP fixed point for BOTS.} 

Our main contribution is to show that in a wide range of parameters and side-information models, the BOTS associated distributional fixed
point equation (known as BP fixed point) has a unique solution. This automatically has several implications.

First, this establishes the desired  `boundary irrelevance' property for the BEC survey, i.e., $h=\bar{h}$:
\begin{equation}\label{eq:bi_1}
\lim_{k \to \infty} H(\sigma_\rho|\sigma_{L_k}, \omega_{T_k}^\epsilon) = \lim_{k \to \infty} H(\sigma_\rho |\omega_{T_k}^\epsilon).
\end{equation}
The latter always implies 
\begin{equation}\label{eq:bi_2}
\lim_{\epsilon \to 1} \lim_{k \to \infty} H(\sigma_\rho |\omega_{T_k}^\epsilon)
=\lim_{k \to \infty} H(\sigma_\rho|\sigma_{L_k}) \,.
\end{equation}
Indeed, one only needs to notice that $\lim_{\epsilon\to 1} \lim_{k \to \infty} H(\sigma_\rho|\sigma_{L_k}, \omega_{T_k}^\epsilon) =\sup_{\epsilon,k} H(\sigma_\rho|\sigma_{L_k}, \omega_{T_k}^\epsilon)
$
and that for every $k$ the latter quantity is continuous in
$\epsilon \in [0,1]$ including at the boundary. 

Further, the presence of the survey allows to convert the absence of leaf-information into the presence of noisy leaf-information, thereby obtaining the robust reconstruction property in the presence and in the absence of the survey \cite{Mossel_2016}.

Property~\eqref{eq:bi_2} is also known in the SBM literature as the condition for ``optimality
of local algorithms'', and was investigated in~\cite{kanade2014global,mossel2015local}. These works build on
the crucial contribution of~\cite{Mossel_2016}, which shows uniqueness of BP fixed point for BOT without survey and
$d\theta^2 > C$, where $C$ is ``large enough'' (see Appendix \ref{apx:mns_const} for our estimates of how large).
Note that since the conditional entropy in~\eqref{eq:bi_2} can be sandwiched between $H(\sigma_\rho| \sigma_{L_k})$ and $H(\sigma_\rho|\omega_{L_k}^\epsilon)$, the result of~\cite{Mossel_2016} implies~\eqref{eq:bi_2}, as indeed observed in~\cite[Prop.
3]{kanade2014global}. However, \cite{kanade2014global}  derives result for the case where $\epsilon\to 1$, relying on \cite{Mossel_2016} for large enough $C$. It also conjectures the
more general~\eqref{eq:bi_1} (for all $d,\theta,\epsilon$ and BEC survey), and our paper validates this conjecture in a wide range of parameters (see Fig.~\ref{fig:dt2}), including for all values of the $d\theta^2 \not \in (1,3.513)$.

Finally,  subsequent work~\cite{mossel2015local}
focuses on the case of $\BSC_\epsilon$ rather than $\BEC_\epsilon$ survey, and also conjectured~\eqref{eq:bi_1} for all $d,\theta,\epsilon$. They 
demonstrate the uniqueness of the BP fixed point in this setting for some range of parameters (which as
$\epsilon \to 1/2$ reduces to $d\theta^2 > C$ for some large enough $C$). Although the method of~\cite{mossel2015local} is an
extension of~\cite{Mossel_2016}, the authors make the  remark ``We note however that the paper~\cite{Mossel_2016} did not consider
side information and the adaptation of the proof is far from trivial.'' This is further expanded in the current paper.

\subsection{Novelty and comparison to the literature}\label{sec:related}
We believe that our proof technique offers the following improvements compared to~\cite{Mossel_2016,mossel2015local}: 
(a) it is much shorter; (b) we do not need to consider large $\theta$, small $d$  and
small $d$ large $\theta$ cases separately; (c) it works simultaneously for $d\theta^2 < 1$ and $d\theta^2 > 3.513$; (d) it
works simultaneously with and without side-information, and the side-information can be any BMS, rather than specifically the BEC or BSC; 
(e) it closes the entire low-SNR case $d\theta^2 <
1$\footnote{There are, however, two related low-SNR results. \cite[Theorem 4.2]{mossel2015local} shows uniqueness of
fixed point for $d\theta<1$ via  a simple contractivity of $F_\theta$
function in the BP recursion~\eqref{eqn:llr-recursion}.\cite[Theorem 3]{kanade2014global} shows~\eqref{eq:bi_2} for
$d\theta^2<1$ as an application of information contraction from~\cite{evans2000}.}, and to the best of our knowledge it
yields the state-of-the-art threshold for the high-SNR case. 

Our main innovation is the information-theoretic point of view: we consider BOTS with or without leaf
observations as two binary input symmetric channels (BMSs) which are related to each other by a property known as 
degradation. This implies a certain inequality between the log-likelihood ratios (LLRs), cf.~\eqref{eq:deg_1}, which we
exploit in the application of the potential method. These key ideas are the content of the
Prop.~\ref{prop:reduce-to-contraction}. On
the more technical side, another innovation is the choice of the potential function as $\phi(r) = e^{-1/2 r}$.

%
%


Concerning the reduction from SBM to BOTs, 
we note first that the reduction in \cite{Mossel_2016} is obtained for the agreement metric. It is easy to navigate between agreement and entropy once on the tree models, but in the SBM, the entropy allows for the chain rule and other properties that lead to the direct reduction detailed previously. On the other hand, \cite{Mossel_2016}, rely on a black-box algorithms that solves weak recovery in order to bring the noisy leaves. Therefore, we are trading the noisy leaves with the survey. In turn, we can exploit the survey to obtain tighter conditions for the boundary irrelevance that lead to part (ii) of Theorem \ref{thm:HSBM-BOT}.

Finally, \cite{10.1145/3055399.3055420} uses a reduction to trees for the entropy that does also not rely on the erasure side-information as described above. In particular, the computation of the SBM entropy is linked to an optimization problem (Theorem 2.2 therein), whose solution corresponds to the dominant BP fixed point on a Galton-Watson tree (Theorem 2.4). 



\section{Results: Boundary Irrelevance and SBM Entropy} \label{sec:bi-and-sbm}

\paragraph{Broadcasting on Trees with Survey (BOTS).}
We start with the standard broadcasting on trees (BOT) setting.
Let $T$ be an infinite tree rooted at $\rho$.
Let $\sigma_\rho \sim \Unif(\{\pm 1\})$ be the root bit and assume that it is broadcast through each edge independently with flip probability $\delta \in (0, \frac 12]$.
For simplicity we use notation $\theta = 1-2\delta$.
Let $L_k$ denote the set of nodes at level $k$, and $T_k$ denote the set of nodes at level $\leq k$ (where the root is at level 0).
Reconstruction on such models consists of recovering the root bit after observing the leaves bits at large depth (\cite{evans2000}).

We consider a slightly different problem, where we have access to some node side-information, or ``survey''.
Specifically, let $W$ be a fixed BMS channel, and for each node $u$ we observe $\omega_u \sim W(\sigma_u)$. We call $(T,\rho,\theta, W)$ a broadcasting instance with survey. We will also denote by $\Delta_W$ the $\Delta$-component of the BMS $W$ (see Appendix \ref{apx:bms} for background on BMS channels).
This setting includes the one in \cite{mossel2015local}, where $W = \BSC_\alpha$, i.e., for each node $u$, $\bP[\omega_u= \sigma_u] = 1 - \bP[\omega_u= -\sigma_u] = 1-\alpha$; and the one in \cite{kanade2014global}, where $W = \BEC_\epsilon$, i.e., for each node the survey reveals the correct label with probability $1-\epsilon$ and an erasure symbol otherwise.
The latter is of particular interest to us because of its application to the computation of the SBM entropy (Theorem \ref{thm:HSBM-BOT}). For clarity, in the case of erasure survey, we denote $\omega_{u}^\epsilon = \BEC_\epsilon(\sigma_u)$.  

\begin{thm} \label{thm:HSBM-BOT}
Let $(X,G) \sim SBM(n,2,a/n,b/n)$. Let $T$ be a Galton-Watson tree with Pois$(\frac{a+b}{2})$ offspring distribution and let $(T,\rho, \frac{a-b}{a+b}, \BEC_\epsilon)$ be a broadcasting instance with erasure survey, and edge flip probability $\frac{b}{a+b}$. Let $\alpha^* \approx 3.513$ be the unique solution in $\mR_{>1}$ to the equation 
      $\exp(-\frac{\alpha-1}2) \alpha = 1$.
The following hold.
\begin{enumerate}[(i)]
    \item For $a,b$ such that $\frac{(a-b)^2}{2(a+b)}\leq 1$ or $\frac{(a-b)^2}{2(a+b)}\geq \alpha^* \approx 3.513$
\begin{align} \label{eq:HSBM}
   \mathcal{H}(a,b) = \lim_{n \rightarrow \infty} \frac{1}{n}H(X|G) = \int_0^1 \lim_{k \to \infty} H(\sigma_\rho|T, \sigma_{L_k}, \omega^\epsilon_{L_k}) d\epsilon.
\end{align}
    \item For any $a,b$ such that $\frac{(a-b)^2}{2(a+b)}\in (1,\alpha^*)$, i.e., inside the gap of part (i), 
    \begin{align} 
  & \lim\inf_{n \rightarrow \infty} \frac{1}{n}H(X|G) =  \int_0^1 \lim_{k \to \infty} H(\sigma_\rho|T, \sigma_{L_k}, \omega^\epsilon_{L_k}) d\epsilon + \xi_{\text{inf}}, \\ 
& \lim\sup_{n \rightarrow \infty} \frac{1}{n}H(X|G) =  \int_0^1 \lim_{k \to \infty} H(\sigma_\rho|T, \sigma_{L_k}, \omega^\epsilon_{L_k}) d\epsilon + \xi_{\text{sup}},
\end{align}
where $0 \leq\xi_{\text{inf}},  \xi_{\text{sup}}\leq 1- \frac{\sqrt{e}}{2} \approx 0.178 $.
    \end{enumerate}
\end{thm}
A crucial ingredient to establish Theorem~\ref{thm:HSBM-BOT} is the following property for BOTS. 

\begin{defin}[Boundary Irrelevance (BI)]
We say that $(T,\rho,\theta, W)$ has the Boundary Irrelevance (BI) property if 
\begin{align} \label{eq:BI}
    \lim_{k\to\infty} I(\sigma_\rho; \sigma_{L_k} |T, \omega_{T_k}) = 0.
\end{align}
which is equivalent to~\eqref{eq:bi_1}.
\end{defin}
In words, the (BI) implies that if we have access to some intermediate node information, the leaves at infinite depth become irrelevant for detecting the root bit.
We focus on regular and Galton-Watson trees with Poisson offspring. We prove the following Theorem in Section \ref{sec:recon-on-tree}.

\begin{thm} \label{thm:BI-allW}
Let $T$ be a $d$-regular tree or a Galton-Watson tree with Poisson$(d)$ offspring distribution, with root vertex $\rho$. Let $W$ be a BMS channel.
If $P_e(W) \ne \frac 12$, and
    \begin{align}
    d \theta^2 \exp(-\frac{(d\theta^2-1)_+}2) Z(W) < 1,\label{eqn:high-snr-tight}
    \end{align}
    where $P_e(W)$ is the probability of error, and $Z(W)$ is the Bhattacharyya coefficient (defined in Definition \ref{defn:information-measures}),
    then (BI) holds for $(T, \rho, \theta, W)$.
In particular, (BI) holds for any $(T, \rho, \theta, W)$ with $d\theta^2 < 1$ or $d\theta^2 > \alpha^*$ (and with $P_e(W) \neq \frac 12$), where $\alpha^* \approx 3.513$ is the unique solution in $\mR_{>1}$ to the equation 
      $\exp(-\frac{\alpha-1}2) \alpha = 1$.
\end{thm}
We remark that~\eqref{eqn:high-snr-tight} is a relaxation of a sharper bound in Prop.~\ref{prop:high-snr} (e.g., for regular trees with $d=2$ (BI) is proven for all cases except $d\theta^2 \in (1,1.62)$). The following corollary lists a few direct consequences of Theorem \ref{thm:BI-allW}.

\begin{corol}\label{cor:BI-conseq}
In the setting of Theorem \ref{thm:BI-allW}, if any of the following is true, then (BI) holds for $(T, \rho, \theta, W)$:
    (i) $Z(W) < \frac{\sqrt e}2\approx 0.824$;
    (ii) $P_e(W) < \frac 12 - \frac 14 \sqrt{4-e} \approx 0.217$;
    (iii) $W = \BEC_\epsilon$ and with $\epsilon < \frac{\sqrt e}2\approx 0.824$.
\end{corol}
\begin{proof}
For (i) we observe that   $  \sup_{\alpha \ge 0} (\alpha \exp(-\frac{\alpha-1}2)) = \frac{2}{\sqrt e}
$. For (ii) we define $p(\Delta) = 2\sqrt{\Delta(1-\Delta)}$ and notice that 
  $
  Z(W) = \bE[p(\Delta_W)] \le p(\bE \Delta_W) = p(P_e(W))
  $
  because the function $p$ is concave.
  So when $P_e(W) < \frac 12 - \frac 14 \sqrt{4-e}$, we have
  $Z(W) < \frac{\sqrt e}{2}$.
(iii) follows from (i).
\end{proof}

Theorem~\ref{thm:BI-allW} is a consequence of the following more general result, that we state informally here (for the
full statement see~ Prop.\ref{prop:uniqueness-bp-fixed-point} in Appendix \ref{app:uniqueness}).

\begin{prop}[Informal, uniqueness of BP fixed point] For BOTS if~\eqref{eqn:high-snr-tight} holds, then the BP
distributional fixed point is unique. For BOT if $d\theta^2 < 1$ or $d\theta^2 > \alpha^*$ then the non-trivial fixed
point is unique.
\end{prop} 
We demonstrate the region of BP-uniqueness from Corollary \ref{cor:BI-conseq} on Figure \ref{fig:dt2}.
	\begin{figure}[ht]
	\begin{center}
		\includegraphics[scale=0.5]{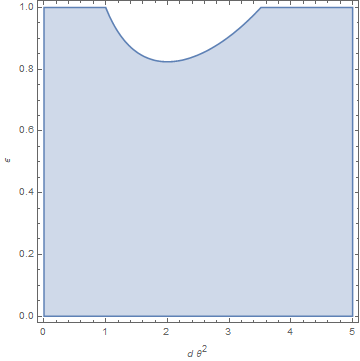}
		\quad
		\includegraphics[scale=0.5]{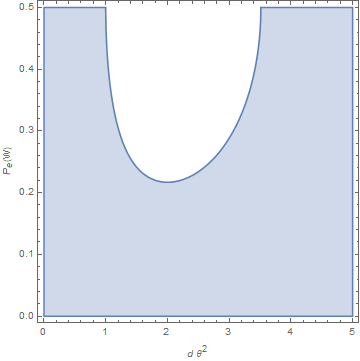}
		\end{center}
		\caption{Left: Region of BP uniqueness for BEC survey from Corollary \ref{cor:BI-conseq}(iii).\\
		Right: Region of BP uniqueness for BMS survey from Corollary \ref{cor:BI-conseq}(ii).}
		\label{fig:dt2}
	\end{figure}
We note that, taking the limit $\epsilon \to 1^-$, Theorem \ref{thm:BI-allW} implies that revealing an (arbitrarily)
small fraction of vertex labels gives the same information about the root bit, as revealing the whole boundary labels at
large distance, even in the reconstruction regime, cf.~\eqref{eq:bi_2}.

\begin{conj}\label{conj:main}
    Let $T$ be a regular tree or a Galton-Watson tree with Poisson offspring distribution, with root vertex $\rho$.
    Then (BI) holds for $(T, \rho, \theta, W)$ for all $0 <\theta< 1$ and all $W$ such that $P_e(W) \neq \frac 12$.
\end{conj}
If Conjecture \ref{conj:main} holds, the proof of Theorem \ref{thm:HSBM-BOT} gives a precise characterisation of $\mathcal{H}(a,b)$, as in~\eqref{eq:HSBM}, in terms of BOTS entropies for the entire range of $a,b$. 

\noindent
We now prove Theorem \ref{thm:HSBM-BOT}, and defer the proof of Theorem \ref{thm:BI-allW} to Section \ref{sec:recon-on-tree}. 
\section{Proof of Theorem \ref{thm:HSBM-BOT}}
  Let us denote $f(\epsilon) = H(X|G,Y^\epsilon) $, where similarly as before $Y^\epsilon$ is a $\BEC_\epsilon$-survey that reveals the true label of each node independently with probability $1-\epsilon$. Note that $f(1) = H(X|G) $. Let us replace the single parameter $\epsilon$ by a set of parameters $\vec \epsilon = (\epsilon_u)_{u \in V(G)}$ (for each vertex $u$, $X_u$ is revealed with probability $1-\epsilon_u$), and let us denote $Y_{\sim u}^\epsilon = \{ Y_v^\epsilon: v \in V(G), v\neq u \}$ and $X_{\sim u} = \{ X_v: v \in V(G), v \neq u\}$. Then \begin{align}
  f(\vec \epsilon) = (1-\epsilon_u) H(X|G, X_u, Y_{\sim u}^\epsilon) + \epsilon_u H(X|G,Y_{\sim u}^\epsilon) \end{align} and by chain rule
\begin{align}
\frac{\partial}{\partial \epsilon_u} f(\vec \epsilon)& = H(X|G, Y_{\sim u}^\epsilon) - H(X|G, X_u, Y_{\sim u}^\epsilon) \\
& = H(X_u, X_{\sim u} | G, Y_{\sim u}^\epsilon) - H(X_{\sim u} | G, X_u, Y_{\sim u}^\epsilon)\\
& = H(X_u | G, Y_{\sim u}^\epsilon).
\end{align}
Then, setting $\epsilon_u = \epsilon$ for all $u \in V(G)$, we get by symmetry
\begin{align}
f'(\epsilon) = \sum_{u \in V(G)} H(X_u | G, Y_{\sim u}^\epsilon)= n H(X_1 | G, Y_{\sim 1}^\epsilon).
\end{align} Thus, by bounded convergence
\begin{align} \label{eq:HSBM-integrand}
\lim_{n \rightarrow \infty} \frac{1}{n}H(X|G) = \int_0^1 \lim_{n \rightarrow \infty} H(X_{1} |G,Y_{\sim 1}^\epsilon) d\epsilon.
\end{align}
Take $k =\frac{\log{n}}{10 \log{2(a+b)}}$ small enough compared to $n$, such that the neighborhood of vertex $1$ at depth $k$ is a tree with high probability (this is for instance proved as Proposition 2 in \cite{MNSplanted15}), and denote such neighborhood by $T_k$. Specifically, w.h.p. $T_k$ is a Galton-Watson tree with Poisson$\left(\frac{a+b}{2}\right)$ offspring distribution, rooted at $1$, and the labels in $X_{T_k}$ are distributed as BOT with flip probability $\frac{b}{a+b}$. Moreover, let $X_{L_k} $ be the vertices at distance exactly $k$ from $1$, and let $Y_{\sim 1, T_k}^\epsilon$ denote the survey on nodes at distance at most $k$ from $1$ (excluding $1$). We bound the integrand by the following:
\begin{align} \label{eq:boundsentropy}
  H(X_1|T_k, Y_{\sim 1, T_k}^\epsilon, X_{L_k}) + o_k(1)  \leq H(X_{1} |G,Y_{\sim 1}^\epsilon) \leq H(X_1 |T_k,  Y_{\sim 1, T_k}^\epsilon).
\end{align}
For the inequality on the right, we simply removed conditioning terms and thus increased the conditional entropy, specifically we ignored any information from the graph or from the survey on nodes at distance $\geq k$ to $1$. The inequality on the left requires the following lemma, that is a direct consequence of Proposition 2 and Lemma 4.7 in \cite{MNSplanted15}.
\begin{lemma} \label{lem:SBM-markov}
$ H(X_{1} |G,Y_{\sim 1}^\epsilon, X_{L_k} ) =  H(X_{1} |T_k,Y_{\sim 1, T_k }^\epsilon, X_{L_k} ) + o_k(1).$
\end{lemma}
In words, Lemma \ref{lem:SBM-markov} states that after conditioning on the leaves, the information coming from the graph outside $T_k$ (including non-edges) becomes negligible, i.e. the model is asymptotically a Markov field. By Theorem \ref{thm:BI-allW}, if $\frac{(a-b)^2}{2(a+b)}\leq1$ or $\frac{(a-b)^2}{2(a+b)}\geq \alpha^* $, then (BI) holds for $(T_k, 1, \frac{a-b}{a+b},\BEC_\epsilon) $, for all $\epsilon <1$, thus the leftmost and the rightmost terms in (\ref{eq:boundsentropy}) are asymptotically equal. This means that the limit in the integrand in~\eqref{eq:HSBM-integrand} exists for all $\epsilon \in(0,1)$, thus (i) holds.

On the other hand, by Corollary \ref{cor:BI-conseq}(iii), for all $\epsilon < \epsilon^* = \frac{\sqrt{e}}{2} \approx 0.824$ and for all $a,b$ the (BI) holds for $(T_k, 1, \frac{a-b}{a+b},\BEC_\epsilon) $. Thus
\begin{align}
  \lim\inf_{n \to \infty}   \frac{1}{n} H(X|G) = \int_0^{\epsilon^*} \lim_{k \to \infty} H(X_1 | T_k,Y_{\sim 1, T_k}^\epsilon , X_{L_k}) d\epsilon + \xi_{\text{inf}},
\end{align}
with 
\begin{align}
\xi_{\text{inf}} &= \int_{\epsilon^*}^1 \lim\inf_{n \to \infty}  H(X_1|G,Y_{\sim1}^\epsilon) - \lim_{k \to \infty}H(X_1|T_k, Y_{\sim1,T_k}^\epsilon, X_{L_k}) d\epsilon \\
&\leq (1-\epsilon^*) \lim_{k \to \infty}I(X_1;X_{L_k} | T_k) \leq (1-\epsilon^*) \approx 0.178 .
\end{align} 
The same holds for $\lim\sup $ and $\xi_{\text{sup}}$.


\section{Proof of Theorem \ref{thm:BI-allW}}\label{sec:recon-on-tree}



Recall the BOTS model defined in Section \ref{sec:bi-and-sbm}.
Let $M_k$ denote the BMS channel $\sigma_\rho\to (\omega_{T_k}, \sigma_{L_k})$
and $\tilde M_k$ denote the BMS channel $\sigma_\rho\to \omega_{T_k}$.
Let $P_{\Delta_k}$ (resp.~$P_{\tilde \Delta_k}$) be the distribution of the $\Delta$-component of the BMS $M_k$ (resp.~$\tilde M_k$).
We prove the following strengthening of Theorem \ref{thm:BI-allW}.
\begin{thm}\label{thm:uniqueness-survey}
  In the setting of Theorem \ref{thm:BI-allW},
   $P_{\Delta_k}$ and $P_{\tilde \Delta_k}$ converge in distribution to the same distribution as $k\to \infty$.
  In particular,
  \begin{align}
    \lim_{k\to \infty} P_e(M_k) &= \lim_{k\to \infty} P_e(\tilde M_k),\\
    \lim_{k\to \infty} C(M_k) &= \lim_{k\to \infty} C(\tilde M_k).
  \end{align}
\end{thm}
Clearly Theorem \ref{thm:uniqueness-survey} implies Theorem \ref{thm:BI-allW}.


\subsection{Belief propagation recursion}\label{sec:bp_rec}
The maximum a posteriori probability (MAP) decoder is the optimal decoder for this reconstruction problem.
It can be implemented using belief propagation (BP) as follows.

For each node $u$, let $L_k(u)$ denote the set of nodes in subtree rooted at $u$ that are at distance $k$ to $u$.
Let $T_k(u)$ denote the set of nodes in subtree rooted at $u$ that are at distance $\le k$ to $u$.
Let $R_{u,k}\in \mR\cup \{\pm \infty\}$ denote the posterior log likelihood ratio given $\omega_{T_k(u)}\cup \sigma_{L_k(u)}$:
\begin{align}
  R_{u,k} = \log \frac{\bP[\sigma_u = + | \omega_{T_k(u)}\cup \sigma_{L_k(u)}]}{ \bP[\sigma_u = - | \omega_{T_k(u)}\cup \sigma_{L_k(u)}]}.
\end{align}

The initial value is
\begin{align}
  R_{u,0} = \sigma_u \cdot \infty. \label{eqn:llr-initial}
\end{align}

Define a function $F_\theta: \mR\cup\{\pm\infty\}\to \mR$ as
\begin{align}
  F_\theta(r) = 2\arctanh(\theta \tanh(\frac 12 r)).\label{eqn:F-theta-defn}
\end{align}
Then by definition of $R_{u,k}$, we have
\begin{align}
  R_{u,k+1} = \sum_{v\in L_1(u)} F_\theta(R_{v,k}) + W_u \label{eqn:llr-recursion}
\end{align}
where $W_u$ is the log likelihood ratio induced by observation,
meaning that
\begin{align}
  W_u = \log \frac{\bP[\sigma_u = + | \omega_u]}{\bP[\sigma_u = - | \omega_u]}.
\end{align}

Using \eqref{eqn:llr-initial}\eqref{eqn:llr-recursion} we are able to compute $R_{\rho,k}$ recursively.

For observation without leaves, let $\tilde R_{u,k}$ denote the posterior log likelihood ratio given $\omega_{T_k(u)}$.
Then $\tilde R_{u,k}$ satisfies the same recursion \eqref{eqn:llr-recursion}, but with a different initial value
\begin{align}
  \tilde R_{u,0} = 0.
\end{align}

Let $M_k(u)$ denote the BMS channel $\sigma_u\to (\omega_{T_k(u)}, \sigma_{L_k(u)})$.
Let $\tilde M_k(u)$ denote the BMS channel $\sigma_u\to \omega_{T_k(u)}$.
Let $\Delta_{u,k}$ and $\tilde \Delta_{u,k}$ denote the corresponding $\Delta$-components
(both are random variables supported on $[0, \frac 12]$).
They relate to log likelihood ratio by the following expression:
\begin{align}\label{eq:delta_to_R}
  |R_{u,k}| = \log \frac{1-\Delta_{u,k}}{\Delta_{u,k}}, \quad |\tilde R_{u,k}| = \log \frac{1-\tilde \Delta_{u,k}}{\tilde \Delta_{u,k}}.
\end{align}

Channel $M_k(u)$ is less degraded than $\tilde M_k(u)$, by forgetting $\sigma_{L_k(u)}$.
So under the canonical coupling\footnote{By canonical coupling we mean the joint distribution induced
by the downward BOTS process followed by the evaluation of $R$ and $\tilde R$ via an upward run of BP.}, we have
\begin{align}\label{eq:deg_1}
  \bE[\Delta_{u,k} | \tilde \Delta_{u,k}] \le \tilde \Delta_{u,k}.
\end{align}
As we will see, the core of our proof is the use of this degradation relationship.

Let $\mu_k^+$ be the distribution of $R_{u,k}$ conditioned on $\sigma_u=+$, and $\tilde \mu_k^+$ be the distribution of $\tilde R_{u,k}$ conditioned on $\sigma_u=+$.
(They do not depend on the choice of $u$.)
Then $\mu_0^+$ is the point measure at $+\infty$, $\tilde \mu_0^+$ is the point measure at $0$.

Both distributions satisfy the same recursion. Consider the equation
\begin{align}
  R_{u,k+1}^{+} = \sum_{v\in L_1(u)} Z_v F_\theta(R_{v,k}^+) + W_u\label{eqn:distribution-recursion}
\end{align}
where $\{Z_v, R_{v,k}^+, W_u : v\in L_1(u)\}$ are independent,
  $Z_v$ are i.i.d.~Bernoulli ($\bP[Z_v=+1] = 1-\bP[Z_v = -1] = 1-\delta$),
   $R_{v,k}^+ \sim \mu_k^+$ (resp.~$\sim \tilde \mu_k^+$),
  and  $W_u$ is distributed as log likelihood ratio corresponding to the survey BMS.

Then $R_{u,k+1}^{+} \sim \mu_{k+1}^+$ (resp.~$\sim \tilde \mu_{k+1}^+$).
\eqref{eqn:distribution-recursion} also holds if we replace all $R_{u,k}^+$ with $\tilde R_{u,k}^+$.

Note that given $\Delta_{u,k}$, $R_{u,k}^+$ is as described by \eqref{eq:R_to_delta}.
The same holds for $\tilde \Delta_{u,k}$ and $\tilde R_{u,k}^+$.

\paragraph{BP distributional fixed point.}
A distribution $\mu$ on $\mathbb{R}\cup\{\pm \infty\}$ is called a BP fixed point of the BOTS $(d,\theta,W)$ if taking $R^+_{i}$ i.i.d.~$\sim \mu$, $i\in[d]$, $Z_i$ and $R_W$ as above results in 
\begin{equation}\label{eqn:bp-fixed-point}
    R^+ = \sum_{1\le i\le d} Z_i F_\theta(R_i^+) + R_W
\end{equation}
having the same distribution $\mu$.
In this work we restrict our attention to symmetric distributions, i.e., distributions associated with BMS channels.
We talk below about the fixed point distribution $P_\Delta$ that is related to $\mu$ via transformation~\eqref{eq:delta_to_R}. Namely, a distribution $P_\Delta$ is a fixed point iff the law $\mu$ of random variable $R^+$ is a fixed point, where $R^+$ is generated via sampling $\Delta \sim P_\Delta$ and then setting
 \begin{align}\label{eq:R_to_delta}
  R^+ = \begin{cases}
    \log \frac{1-\Delta}{\Delta}, & \text{w.p.~} 1-\Delta,\\
    -\log \frac{1-\Delta}{\Delta}, & \text{w.p.~} \Delta.
    \end{cases}
  \end{align}
Similarly, we define the BP fixed point for the BOTS $(\mathrm{Poi}(d), \theta, W)$ where in~\eqref{eqn:bp-fixed-point} $d$ is replaced with $b\sim \mathrm{Poi}(d)$.

\subsection{Contraction of potential function}
The technical part of our proof is contraction of certain potential functions.
The next proposition shows the kind of contraction result we need.
\begin{prop}\label{prop:reduce-to-contraction}
  Let $\phi: \mR\cup\{\pm \infty\} \to \mR\cup\{\pm \infty\}$ be a function such that the function
  $g : [0, \frac 12] \to \mR\cup\{\pm \infty\}$ defined as 
  \begin{align}
    g(\Delta) = (1-\Delta)\phi(\log \frac{1-\Delta}{\Delta}) + \Delta \phi(-\log \frac{1-\Delta}{\Delta})
  \end{align}
  is decreasing and $\alpha$-strongly convex for some $\alpha>0$.
  If 
  \begin{align}
    \lim_{k\to \infty} \bE[ \phi(R_{\rho,k}^+) - \phi(\tilde R_{\rho,k}^+)] = 0,\label{eqn:potential-contraction}
  \end{align}
  then under the canonical coupling, 
  \begin{align}
    \lim_{k\to \infty} \bE (\Delta_{\rho,k} - \tilde\Delta_{\rho,k})^2 = 0.\label{eqn:l2-contraction}
  \end{align}
\end{prop}
\begin{proof}
  Because $g$ is $\alpha$-strongly convex, we have
  \begin{align}
    g(\Delta_{\rho,k}) - g(\tilde \Delta_{\rho,k}) \ge g'(\tilde \Delta_{\rho,k}) (\Delta_{\rho,k} - \tilde \Delta_{\rho,k}) + \frac{\alpha}{2} (\Delta_{\rho,k} - \tilde \Delta_{\rho,k})^2.
  \end{align}
  So
  \begin{align}
    \bE[\phi(R_{\rho,k}^+) - \phi(\tilde R_{\rho,k}^+)] &= \bE_{\tilde \Delta_{\rho,k}}  \bE[\phi(R_{\rho,k}^+) - \phi(\tilde R_{\rho,k}^+) | \tilde \Delta_{\rho,k}] \nonumber \\
    & = \bE_{\tilde \Delta_{\rho,k}}  \bE[g(\Delta_{\rho,k}) - g(\tilde \Delta_{\rho,k}) | \tilde \Delta_{\rho,k}] \nonumber\\
    & \ge \bE_{\tilde \Delta_{\rho,k}}  \bE[g'(\tilde \Delta_{\rho,k}) (\Delta_{\rho,k} - \tilde \Delta_{\rho,k}) + \frac{\alpha}{2} (\Delta_{\rho,k} - \tilde \Delta_{\rho,k})^2 | \tilde \Delta_{\rho,k}] \nonumber\\
    & = \bE_{\tilde \Delta_{\rho,k}}  [g'(\tilde \Delta_{\rho,k}) (\bE[\Delta_{\rho,k} | \tilde \Delta_{\rho,k}]- \tilde \Delta_{\rho,k})]
    + \frac{\alpha}{2} \bE (\Delta_{\rho,k} - \tilde \Delta_{\rho,k})^2 \nonumber\\
    &\ge \frac{\alpha}{2} \bE (\Delta_{\rho,k} - \tilde \Delta_{\rho,k})^2. \label{eqn:potential-nonneg}
  \end{align}
  Clearly \eqref{eqn:potential-contraction} implies \eqref{eqn:l2-contraction}.
\end{proof}
Note that \eqref{eqn:potential-nonneg} also shows that $\bE[\phi(R_{\rho,k}^+) - \phi(\tilde R_{\rho,k}^+)]$ is always non-negative.

We choose the potential function to be $\phi(r) = -\exp(-\frac 12 r)$.
The function $g$ is
  $g(\Delta) = -2\sqrt{\Delta(1-\Delta)}.$
One can check that $g$ is decreasing and $4$-strongly convex on $[0, \frac 12]$.
\begin{prop}\label{prop:high-snr}
  Assume that we have a non-trivial survey channel.
  Let
  \begin{align}
  C_1 = C_1(d, \theta, W) = d \theta^2 (1-\frac{(d\theta^2-1)_+}{d-1})^{\frac{d-1}2} Z(W).\label{eqn:high-snr-defn-C1}
  \end{align}
  For regular trees, under the canonical coupling, for any $\epsilon>0$, there exists $k^*$ such that for all $k\ge k^*$,
  \begin{align}
    \bE[\exp(-\frac 12 \tilde R_{\rho,k+1}^+) - \exp(-\frac 12 R_{\rho,k+1}^+)] \le (1+\epsilon) C_1 \bE[\exp(-\frac 12 \tilde R_{\rho,k}^+) - \exp(-\frac 12 R_{\rho,k}^+)].
  \end{align}
  In particular, if $C_1 < 1$, then \eqref{eqn:potential-contraction} holds.
  
  For Galton-Watson trees with Poisson offspring distribution, the same holds with $C_1$ replaced by
  \begin{align}
    C_2 = C_2(d, \theta, W) = d \theta^2 \exp(-d (1-\sqrt{1-\frac{(d\theta^2-1)_+}{d}})) Z(W).\label{eqn:high-snr-defn-C2}
  \end{align}
\end{prop}
Proof of Proposition \ref{prop:high-snr} is deferred to Appendix \ref{app:high-snr}.

Proposition \ref{prop:reduce-to-contraction} and \ref{prop:high-snr} complete the proof of Theorem \ref{thm:uniqueness-survey}, because for $i=1,2$, we have
\begin{align}
    C_i \le d \theta^2 \exp(-\frac{(d\theta^2-1)_+}2) Z(W).
\end{align}

\section{Other results}

\paragraph{Weak spatial mixing.}
BOT (without survey) is an example of the Ising model. As it is typical for such models, at high temperature (i.e. $d\theta \le 1$) it exhibits the property known as weak spatial mixing (WSM): enforcing a (far away) boundary condition does not affect the distribution of spins. This property disappears at low temperatures ($d\theta >1$), but what is surprising is that there is a range of parameters ($d\theta > 1$ but $d\theta^2 < 1$) in which there is no WSM, but reconstruction is still impossible~\cite{bleher1995purity}.

Now, the BOTS model can be thought of as an example of an Ising spin glass system: one first generates the survey and then, treating the survey as quenched randomness, considers an Ising model with external fields corresponding to survey. The question we ask is whether in this spin-glass type model we still have that (in the limit of vanishing survey) the threshold for WSM appears at $d\theta=1$. Some partial results towards this are contained in Appendix~\ref{app:wsm}. We mention that for BEC survey we were not able to show this.

\paragraph{Boundary irrelevance (BI) on amenable graphs.} So far we studied (BI) property~\eqref{eq:BI} for trees, but it can also be defined for general graphs as follows.

Let $G = (V, E)$ be an infinite graph. Consider the Spin Synchronization model, where we have i.i.d.~random variables $X_v \sim \Unif(\{\pm 1\})$ for $v\in V$; for each edge $uv\in E$, we observe a random variable $Y_{uv}\sim \BSC_\delta(X_u X_v)$, and we denote $\theta = 1-2\delta$. Conditioned on the $X$ variables, the $Y$ variables are mutually independent. In addition to the edge variables, we may observe surveys at each node: for $v\in V$, we have $\omega_v\sim W(X_v)$, with $W$ being a fixed BMS channel. In this Section we consider $\BEC_\epsilon$ survey.

Let $o \in V$ be a vertex. Let $B_n(o)$ be the set of nodes with distance $\le n$ to $o$, and $\partial B_n(o)$ be the set of nodes at distance $n$ to $o$.
We use notation $X_{\partial B_n(o)}$ for the set $\{X_v : v\in \partial B_n(o)\}$ and
notation $Y_{B_n(o)}$ for $\{Y_{uv} : uv\in E, u\in B_n(o), v\in B_n(o)\}$.
We say the model $(G, o, \theta, W)$ has the (BI) property if 
\begin{align} \label{eq:BI-SOG}
    \lim_{n\to\infty} I(X_o; X_{\partial B_n(o)} | Y_{B_n(o)}, \omega_{B_n(o)}) = 0.
\end{align}
In Appendix~\ref{app:amenable} we show, by applying results of~\cite{alaoui2019computational}, that (BI) holds for all amenable graphs and survey channel being $\BEC_\epsilon$. The definition of such graphs appears therein, but in a nutshell, it requires the boundary of any subset $S\subset V$ to be negligible compared to $|S|$. 

\section*{Acknowledgement}
Part of this work was supported by the NSF CAREER Award CCF-1552131.

\bibliographystyle{alpha}
\bibliography{reference}

\appendix
\section{Preliminaries on BMS channels}\label{apx:bms}
In this section we give necessary preliminaries on BMS channels. Most material in this appendix can be found in e.g., \cite[Chapter 4]{richardson2008modern}.
\begin{defin}
A channel $P: \{\pm 1\} \to \Y$ is called a BMS (Binary Memoryless Symmetric) channel if there exists a measurable involution $\sigma: \Y \to \Y$ such that
\begin{align}
P(\sigma^{-1}(E)|+) = P( E| -)
\end{align}
for all measurable sets $E\subseteq \Y$.
\end{defin}
\begin{lemma}
Every BMS channel $P$ is equivalent to a channel $X\to (\Delta, Z)$,
where $\Delta\in [0, \frac 12]$ is independent of $X$, and $P_{Z|\Delta,X} = \BSC_\Delta(X)$.
\end{lemma}
In the setting of the above lemma, we call the channel $X\to (\Delta,Z)$ the standard form of $P$,
and call $\Delta$ the $\Delta$-component of $P$.
\begin{defin}
Let $P : \{\pm 1\}\to \Y$ and $Q : \{\pm 1\}\to \Z$ be two BMS channels.
We say $P$ is more degraded than $Q$ (denoted $P \le_{\deg} Q$),
if there exists a channel $R : \Z \to \Y$ such that $P = R\circ Q$.
\end{defin}
\begin{lemma}
Let $P$ and $Q$ be two BMS channels.
Let $\Delta$ be the $\Delta$-component of $P$ and $\tilde \Delta$ be the $\Delta$-component of $Q$.
Then $P\le_{\deg} Q$ if and only if there exists a coupling between $\Delta$ and $\tilde \Delta$
so that
\begin{align}
\bE[\Delta | \tilde \Delta] \le \tilde \Delta
\end{align}
for all $\tilde \Delta\in [0, \frac 12]$ for which LHS exists.
\end{lemma}
\begin{defin}\label{defn:information-measures}
  Let $P$ be a BMS channel and $\Delta$ be the $\Delta$-component of $P$. We define the following quantities.
  \begin{align}
    P_e(P) &= \bE \Delta, \tag{probability of error}\\
    C(P) & = \bE [\log 2 + \Delta \log \Delta + (1-\Delta)\log(1-\Delta)], \tag{capacity}\\
    C_{\chi^2}(P) &= \bE[(1-2\Delta)^2], \tag{$\chi^2$-capacity}\\
    Z(P) & = \bE[2\sqrt{\Delta(1-\Delta)}]. \tag{Bhattacharyya coefficient}
  \end{align}
\end{defin}
By definition, $P_e(P) \in [0, \frac 12]$, $C(P) \in [0, \log 2]$, $C_{\chi^2}(P)\in [0, 1]$, $Z(P)\in [0, 1]$.

\begin{lemma}\label{lemma:deg-info-measure}
  If $P\le_{\deg} Q$, then the following holds:
  \begin{align}
  P_e(P)  \ge P_e(Q),\quad
  C(P) \le C(Q),\quad
      C_{\chi^2}(P) \le C_{\chi^2}(Q),\quad
      Z(P) \ge Z(Q).
  \end{align}
\end{lemma}
\begin{lemma}[{Restatement of \cite[Lemma 4.2(iii)]{evans2000}}] \label{lemma:chi2-var}
  Let $P: X\to Y$ be a BMS channel with $Y$ a real variable, and with involution $Y\mapsto -Y$.
  Then
  \begin{align}
    C_{\chi^2}(P) \ge \frac{(\bE^+ Y)^2}{\Var(Y)}.
  \end{align}
\end{lemma}
\begin{proof}
  Let $X\to (\Delta,Z)$ be the equivalent standard form of $P$.
  By Cauchy-Schwarz, we have
  \begin{align}
    \bE^+[(1-2\Delta)^2] \bE^+[Y^2] \ge (\bE^+[(1-2\Delta)|Y|])^2 = (\bE^+ Y)^2.
  \end{align}
  This is equivalent to the desired result.
\end{proof}

\section{Proof of Proposition \ref{prop:high-snr}} \label{app:high-snr}
  Let us first deal with the regular tree case.
  Let $u$ be a vertex and $v_1,\ldots, v_d$ be its children.
  Let $R_{v_1,k}^+,\ldots, R_{v_d,k}^+$ be i.i.d.~$\sim \mu_k^+$,
  and $\tilde R_{v_1,k}^+,\ldots, \tilde R_{v_d,k}^+$ be i.i.d.~$\sim \tilde \mu_k^+$.
  Define $R_{u,k+1}^{+}$ and $\tilde R_{u,k+1}^{+}$ using \eqref{eqn:distribution-recursion}.
  Furthermore, for $0\le i\le d$,
  define $R_{u,i,k+1}^+$ as
  \begin{align}
    R_{u,i,k+1}^+ = \sum_{1\le j\le i} Z_j F_\theta(\tilde R_{v_j,k}^+) + \sum_{i+1\le j\le d} Z_j F_\theta(R_{v_j,k}^+) + W_u.
  \end{align}
  That is, $R_{u,0,k+1}^+ = R_{u,k+1}^+$, and $R_{u,d,k+1}^+ = \tilde R_{u,k+1}^+$.
  
  For $1\le i\le d$ and $k$ large enough, let us prove that
  \begin{align}
    \bE [\exp(-\frac 12 R_{u,i,k+1}^+) - \exp(-\frac 12 R_{u,i-1,k+1}^+)]
    \le (1+\epsilon) \frac{C_1}{d} \bE[\exp(-\frac 12 \tilde R_{v_1,k}^+) - \exp(-\frac 12 R_{v_1,k}^+)]\label{eqn:high-snr-key}
  \end{align}
  where $C_1$ is defined in \eqref{eqn:high-snr-defn-C1}.
  We prove that \eqref{eqn:high-snr-key} is true even if conditioned on $\tilde \Delta_{v_i,k}$.
  For $\Delta \in [0, \frac 12]$, define 
  \begin{align}
    G(\Delta) = \bE[ \exp(-\frac 12 R_{u,i,k+1}^+) - (1+\epsilon) \frac{C_1}{d} \exp(-\frac 12 \tilde R_{v_i,k}^+)| \tilde \Delta_{v_i,k} = \Delta].
  \end{align}
  Define $p(\Delta) = - g(\Delta) = 2\sqrt{\Delta(1-\Delta)}$ so that we work with non-negative numbers. So
  \begin{align}
    \bE[\exp(-\frac 12 \tilde R_{v_i,k}^+) | \tilde \Delta_{v_i,k} = \Delta] = p(\Delta).
  \end{align}
  Then
  \begin{align}
    &\bE[\exp(-\frac 12 R_{u,i,k+1}^+) | \tilde \Delta_{v_i,k} = \Delta] \nonumber\\
    &= \prod_{1\le j \le i-1} \bE[\exp(-\frac 12 Z_j F_\theta(\tilde R_{v_j,k}^+))] \cdot
    \prod_{i+1\le j\le d} \bE[\exp(-\frac 12 Z_j F_\theta( R_{v_j,k}^+))] \nonumber \\
    &\cdot 
    \bE[\exp(-\frac 12 Z_i F_\theta(\tilde R_{v_i,k}^+)) | \tilde \Delta_{v_i,k} = \Delta]
    \cdot
    \bE[\exp(-\frac 12 W_u)]. \label{eqn:high-snr-G}
  \end{align}
  Let us examine $\bE[\exp(-\frac 12 Z_i F_\theta(\tilde R_{v_i,k}^+)) | \tilde \Delta_{v_i,k} = \Delta]$.
  We can compute that
  \begin{align}
  \exp(-\frac 12 Z_i F_\theta(\tilde R_{v_i,k}^+)) = \left\{ \begin{array}{ll} \exp(-\frac 12 \log \frac{1- \Delta * \delta}{ \Delta * \delta}), & \text{w.p.~} 1- \Delta * \delta, \\ \exp(+\frac 12 \log \frac{1-\Delta * \delta}{ \Delta * \delta}), &  \text{w.p.~} \Delta * \delta, \end{array}\right.
  \end{align}
  where we use notation $\delta_1 * \delta_2 = \delta_1(1-\delta_2) + \delta_2(1-\delta_1)$).
  So 
  \begin{align}
    \bE[\exp(-\frac 12 Z_i F_\theta(\tilde R_{v_i,k}^+)) | \tilde \Delta_{v_i,k} = \Delta] = \bE[p(\Delta * \delta)].
  \end{align}
  Similarly,
  \begin{align}
    \bE[\exp(-\frac 12 Z_j F_\theta(\tilde R_{v_j,k}^+))] &= \bE [p(\tilde \Delta_{v_1,k} * \delta)],\\
    \bE[\exp(-\frac 12 Z_j F_\theta(R_{v_j,k}^+))] &= \bE [p(\Delta_{v_1,k} * \delta)].
  \end{align}
  Finally,
  \begin{align}
  \bE[\exp(-\frac 12 W_u)] = \bE[p(\Delta_W)] = Z(W).
  \end{align}
  So from \eqref{eqn:high-snr-G} we get
  \begin{align}
  \bE[\exp(-\frac 12 R_{u,i,k+1}^+) | \tilde \Delta_{v_i,k} = \Delta]
  =
  \bE [p(\tilde \Delta_{v_1,k} * \delta)]^{i-1} \bE[p(\Delta_{v_1,k} * \delta)]^{d-i}
    p(\Delta * \delta) Z(W).
  \end{align}

  So
  \begin{align}
    G''(\Delta) &= \bE [p(\tilde \Delta_{v_1,k} * \delta)]^{i-1} \bE[p(\Delta_{v_1,k} * \delta)]^{d-i}
     Z(W) \frac{d^2}{d \Delta^2} p(\Delta * \delta)  \nonumber\\
     &- (1+\epsilon) \frac{C_1}{d} p''(\Delta).\label{eqn:high-snr-Gpp}
  \end{align}
  
  Let us bound each factor.
  \begin{align}
    p(\Delta * \delta)
    = 2\sqrt{(\Delta * \delta)(1-\Delta * \delta)} 
    = \sqrt{1-\theta^2(1-2\Delta)^2}.\label{eqn:high-snr-phi-star}
  \end{align}
  So
  \begin{align}
    \bE [p(\tilde \Delta_{v_1,k} * \delta)]
    = \bE[\sqrt{1-\theta^2 (1-2\tilde \Delta_{v_1,k})^2}] 
     \le \sqrt{1-\theta^2 \bE (1-2\tilde \Delta_{v_1,k})^2}.\label{eqn:bhatta-to-chi2}
  \end{align}
  By Proposition \ref{prop:high-snr-chi2}, for any $\epsilon' > 0$, for $k$ large enough, we have
  \begin{align}
  \bE[(1-2\tilde \Delta_{v_1,k})^2] \ge (\frac{d\theta^2-1}{(d-1)\theta^2}-\epsilon')_+.
  \end{align}
  So
  \begin{align}
    \bE [p(\tilde \Delta_{v_1,k} * \delta)] \le
    \sqrt{1-\theta^2 (\frac{d\theta^2-1}{(d-1)\theta^2}-\epsilon')_+}.
  \end{align}

  Similarly, 
  \begin{align}
    \bE [p(\Delta_{v_1,k} * \delta)] \le 
    \sqrt{1-\theta^2 (\frac{d\theta^2-1}{(d-1)\theta^2}-\epsilon')_+}.
  \end{align}
  Note that $p$ is strictly concave on $[0, \frac 12]$, and $p'(\frac 12) = 0$.
  So
  \begin{align}
    \frac{d^2}{d \Delta^2} p(\Delta * \delta) \ge \theta^2 p''(\Delta).\label{eqn:high-snr-phi-star-pp}
  \end{align}
  
  So \eqref{eqn:high-snr-Gpp} gives
  \begin{align}
    G''(\Delta)& \ge (1-\theta^2 (\frac{d\theta^2-1}{(d-1)\theta^2}-\epsilon')_+ )^{\frac{d-1}2} \theta^2 p''(\Delta) Z(W) -  (1+\epsilon)\frac{C_1}{d} p''(\Delta) \nonumber\\
    & = ((1-\theta^2 (\frac{d\theta^2-1}{(d-1)\theta^2}-\epsilon')_+ )^{\frac{d-1}2} \theta^2Z(W) - (1+\epsilon)\frac{C_1}{d}) p''(\Delta).
  \end{align}
  Note that
  \begin{align}
  \lim_{\epsilon'\to 0} (1-\theta^2 (\frac{d\theta^2-1}{(d-1)\theta^2}-\epsilon')_+ )^{\frac{d-1}2}
   = (1-\frac{(d\theta^2-1)_+}{d-1} )^{\frac{d-1}2}.
  \end{align}
  So we can take $\epsilon'>0$ small enough so that
  \begin{align}
  (1-\theta^2 (\frac{d\theta^2-1}{(d-1)\theta^2}-\epsilon')_+ )^{\frac{d-1}2} < (1+\epsilon) (1-\frac{(d\theta^2-1)_+}{d-1} )^{\frac{d-1}2}.
  \end{align}
  So for $k$ large enough, $G''(\Delta)\le 0$ for all $\Delta\in [0, \frac 12]$ and $G(\Delta)$ is convex.
  Also, 
  \begin{align}
    G'(\frac 12) &= \bE [p(\tilde \Delta_{v_1,k} * \delta)]^{i-1} \bE[p(\Delta_{v_1,k} * \delta)]^{d-i}
     Z(W) \frac{d}{d\Delta}|_{\Delta=\frac 12}p(\Delta * \delta) \nonumber\\ 
     &- (1+\epsilon)\frac{C_1}{d} p'(\frac 12) \\
    &= 0.
  \end{align}
  So $G'$ is non-positive, thus $G$ is decreasing on $[0, \frac 12]$.
  Because $M_k (v_i)$ (BMS corresponding to $R_{v_i,k}^+$) is less degraded than $\tilde M_k(v_i)$ (BMS corresponding to $\tilde R_{v_i,k}^+$),
  we get \eqref{eqn:high-snr-key}.
  
  For Galton-Watson trees with Poisson offspring distribution, the proof is very similar to, and slightly more involved than the regular case.
  Let $u$ be a vertex.
  Let $R_{v_1,k}^+, R_{v_2,k}^+,\ldots$ be i.i.d.~$\sim \mu_k^+$,
  and $\tilde R_{v_1,k}^+,\tilde R_{v_2,k}^+,\ldots$ be i.i.d.~$\sim \tilde \mu_k^+$.
  Let $b\sim \Po(d)$ and $v_1,\ldots, v_b$ be the children of $u$.
  For $i\ge 0$, define
  \begin{align}
    R_{u,i,k+1}^+ = \sum_{1\le j\le \min\{i, b\}} Z_j F_\theta(\tilde R_{v_j,k}^+) + \sum_{i+1\le j\le b} Z_j F_\theta(R_{v_j,k}^+) + W_u.
  \end{align}
  For $i \ge 1$, let us prove that
  \begin{align}
    \bE [\exp(-\frac 12 R_{u,i,k+1}^+) - \exp(-\frac 12 R_{u,i-1,k+1}^+)] &
    \le c_i \bE[\exp(-\frac 12 \tilde R_{v_1,k}^+) - \exp(-\frac 12 R_{v_1,k}^+)].\label{eqn:high-snr-poisson-key}
  \end{align}
  where $c_i$ are constants to be chosen later.
  Define
  \begin{align}
    G_i(\Delta) = \bE[ \exp(-\frac 12 \tilde R_{u,i,k+1}^+) - c_i \exp(-\frac 12 \tilde R_{v_i,k}^+)| \tilde \Delta_{v_i,k} = \Delta].
  \end{align}
  Let us prove that $G_i$ is decreasing and convex on $[0, \frac 12]$.
  Similarly to \eqref{eqn:high-snr-Gpp}, we have
  \begin{align}
    G_i''(\Delta) = \bE_b [\mathbbm{1}_{b \ge i} \bE [p(\tilde \Delta_{v_1,k} * \delta)]^{i-1} \bE[p(\Delta_{v_1,k} * \delta)]^{b-i} Z(W) \frac{d^2}{d \Delta^2}p(\Delta * \delta)] - c_i p''(\Delta).\label{eqn:high-snr-poisson-Gpp}
  \end{align}
  Let us study each term in \eqref{eqn:high-snr-poisson-Gpp}.
  By \eqref{eqn:high-snr-phi-star} and Proposition \ref{prop:high-snr-chi2},
  for any $\epsilon'>0$, for $k$ large enough, we have
  \begin{align}
    \bE [p(\tilde \Delta_{v_1,k} * \delta)] \le \sqrt{1-\theta^2 \bE (1-2\tilde \Delta_{v_1,k})^2}
     \le \sqrt{1-\theta^2 (\frac{d\theta^2-1}{d \theta^2}-\epsilon')_+}.
  \end{align}
  Similarly, 
  \begin{align}
    \bE [p(\Delta_{v_1,k} * \delta)] \le \sqrt{1-\theta^2 (\frac{d\theta^2-1}{d \theta^2}-\epsilon')_+}.
  \end{align}
  \eqref{eqn:high-snr-phi-star-pp} still holds in the Poisson case.
  So \eqref{eqn:high-snr-poisson-Gpp} gives
  \begin{align}
    G_i''(\Delta) \ge (\bE_b [\mathbbm{1}_{b \ge i} (1-\theta^2 (\frac{d\theta^2-1}{d \theta^2}-\epsilon')_+)^{\frac{b-1}2}] \theta^2 Z(W) - c_i) p''(\Delta).
  \end{align}
  We can take 
  \begin{align}
    c_i = \bE_b [\mathbbm{1}_{b \ge i}(1-\theta^2 (\frac{d\theta^2-1}{d \theta^2}-\epsilon')_+)^{\frac{b-1}2}] \theta^2 Z(W)\label{eqn:high-snr-poisson-c}
  \end{align}
  so that $G_i''(\Delta)\ge 0$ for all $i\ge 1$ and $\Delta\in [0, \frac 12]$.
  Also,
  \begin{align}
    G_i'(\frac 12) &= \bE_b [\mathbbm{1}_{b \ge i} \bE [p(\tilde \Delta_{v_1,k} * \delta)]^{i-1} \bE[p(\Delta_{v_1,k} * \delta)]^{b-i} Z(W)  \frac{d}{d\Delta}|_{\Delta=\frac 12}p(\Delta * \delta) ]\nonumber\\
    &- c_i p'(\frac 12)\\
    & = 0.
  \end{align}
  So $G_i$ is decreasing.
  
  By summing up \eqref{eqn:high-snr-poisson-key} for $i\ge 1$, we get
  \begin{align}
    \bE[\exp(-\frac 12 \tilde R_{u,k+1}^+) - \exp(-\frac 12 R_{u,k+1}^+)]
    \le (\sum_{i\ge 1} c_i) \bE[\exp(-\frac 12 \tilde R_{v_1,k}^+) - \exp(-\frac 12 R_{v_1,k}^+)].
  \end{align}
  By \eqref{eqn:high-snr-poisson-c}, we have
  \begin{align}
    \sum_{i\ge 1} c_i &= \theta^2 \bE_b [ \mathbbm{1}_{b \ge i}(1-\theta^2 (\frac{d\theta^2-1}{d \theta^2}-\epsilon')_+)^{\frac{b-1}2}] Z(W) \\
    & \le d \theta^2 \exp(-d (1-\sqrt{1-\theta^2 (\frac{d\theta^2-1}{d \theta^2}-\epsilon')_+} )) Z(W).
  \end{align}
  We can take $\epsilon'>0$ small enough so that
  \begin{align}
  \exp(-d (1-\sqrt{1-\theta^2 (\frac{d\theta^2-1}{d \theta^2}-\epsilon')_+} )) < (1+\epsilon) \exp(-d (1-\sqrt{1-\frac{(d\theta^2-1)_+}{d}} )).
  \end{align}
  This finishes the proof for the Poisson tree case.


\section{$\chi^2$-capacity of broadcasting-on-tree channels}
\begin{prop}\label{prop:high-snr-chi2}
Consider the Broadcasting on Trees model defined in Section \ref{sec:bi-and-sbm}, with the following observation models:
\begin{itemize}
    \item $M^1_k: \sigma_\rho \to \nu_{L_k}$, where $\nu_v \sim \BSC_\eta(\sigma_v)$;
    \item $M^2_k: \sigma_\rho \to (\sigma_{L_k}, \omega_{T_k})$.
    \item $M^3_k: \sigma_\rho \to \sigma_{L_k}$;
    \item $M^4_k: \sigma_\rho \to \omega_{T_k}$ with non-trivial survey channel $W$;
    \item $M^5_k: \sigma_\rho \to \omega_{L_k}$ with non-trivial survey channel $W$.
\end{itemize}
For each of the above channels, we have
\begin{itemize}
    \item If we work with regular trees, then
    \begin{align}
    \lim_{k\to \infty} C_{\chi^2}(M_k) \ge \frac{(d\theta^2-1)_+}{\theta^2(d-1)}.
    \end{align}
    \item If we work with Galton-Watston trees with Poisson offspring distribution, then
    \begin{align}
    \lim_{k\to \infty} C_{\chi^2}(M_k) \ge \frac{(d\theta^2-1)_+}{d\theta^2}.
    \end{align}
\end{itemize}
\end{prop}
\begin{proof}
  The $\chi^2$-capacity is always non-negative, so the $d\theta^2 \le 1$ case is automatic.
  In the following we assume $d\theta^2 > 1$.

  First we observe that all $M^i_k$'s are less degraded than $M^1_k$ for some suitable choice of $\eta$.
  This is obvious for $i=2,3$.
  Clearly $M^4_k$ is less degraded than $M^5_k$.
  That $M^5_k\le_{\deg} M^1_k$ follows from \cite[Lemma 2, 3]{roozbehani2019low}, where we can take $\eta = P_e(W)$.
  So by Lemma \ref{lemma:deg-info-measure}, we only need to prove the result for $M^1_k$.
  
  To apply Lemma \ref{lemma:chi2-var}, we need to find a BMS channel more degraded than $M^1_k$ which takes value in $\mR$.
  One natrual choice is the majority decoder.
  We define
  \begin{align}
      S_k = \sum_{v\in L_k} \nu_v.\label{eqn:majority-decoder}
  \end{align}
  Then the channel $\sigma_\rho \to S_k$ is clearly more degraded than $M^1_k$.
  We apply Proposition \ref{prop:high-snr-majority} to conclude.
\end{proof}

\begin{prop}\label{prop:high-snr-majority}
  Assume $d\theta^2 > 1$. Consider the channel $\sigma_\rho \to S_k$ defined in \eqref{eqn:majority-decoder}.

  For regular trees,
  \begin{align}
    \lim_{k\to \infty} \frac{\Var^+ S_k}{(\bE^+ S_k)^2} = \frac{1-\theta^2}{d\theta^2-1}.
  \end{align}
  For Galton-Watson trees with Poisson offspring,
  \begin{align}
    \lim_{k\to \infty} \frac{\Var^+ S_k}{(\bE^+ S_k)^2} = \frac{1}{d\theta^2-1}.
  \end{align}
\end{prop}
\begin{proof}
  The regular tree case is proved in \cite[Lemma 3.4, 3.5]{Mossel_2016}.
  (Note that the expression for $\lim_{k\to \infty} \frac{\Var^+ S_k}{(\bE^+ S_k)^2}$ on top of \cite[pg.~2224]{Mossel_2016} is incorrect.)

  Let us focus on the Poisson tree case.
  It is easy to see that 
  \begin{align}
    \bE^+ S_k = (1-2\eta) (d\theta)^k.\label{eqn:high-snr-expectation}
  \end{align}
  Let $\rho$ be the root, and $v_1,\ldots, v_b$ be its children.
  By variance decomposition, we have
  \begin{align}
    \Var^+ S_{\rho,k+1} &= \Var^+ \bE[S_{\rho,k+1} | b] + \bE_b \Var^+(\bE[S_{\rho, k+1} | b, \sigma_{v_1},\ldots, \sigma_{v_b}] | b) \nonumber\\
    &+ \bE \Var^+(S_{\rho, k+1} | b, \sigma_{v_1},\ldots, \sigma_{v_b}). \label{eqn:high-snr-var}
  \end{align}
  Let us compute each summand.
  \begin{align}
    \Var^+ \bE[S_{\rho,k+1} | b] = \Var^+ (b \theta (1-2\eta) (d\theta)^k) = d \theta^2 (1-2\eta)^2 (d\theta)^{2k}.\label{eqn:high-snr-var-p1}
  \end{align}
  \begin{align}
    &\bE_b \Var^+(\bE[S_{\rho, k+1} | b, \sigma_{v_1},\ldots, \sigma_{v_b}] | b) \nonumber\\
    & = \bE_b \Var^+ (\sum_{i\in [b]} \sigma_{v_i} (1-2\eta) (d\theta)^k | b) \nonumber\\
    & = \bE_b [b (1-\theta^2) (1-2\eta)^2 (d\theta)^{2k}] \nonumber\\
    & = d (1-\theta^2) (1-2\eta)^2 (d\theta)^{2k}.\label{eqn:high-snr-var-p2}
  \end{align}
  \begin{align}
    \bE \Var^+(S_{\rho, k+1} | b, \sigma_{v_1},\ldots, \sigma_{v_b})
     = \bE_b [b \sum_{i\in [b]} \Var^+ S_{v_i,k}]
     = d \Var^+ S_{\rho,k}.\label{eqn:high-snr-var-p3}
  \end{align}
  Plugging \eqref{eqn:high-snr-var-p1}\eqref{eqn:high-snr-var-p2}\eqref{eqn:high-snr-var-p3} into \eqref{eqn:high-snr-var},
  we get 
  \begin{align}
    \Var^+ S_{\rho,k+1} = d (1-2\eta)^2 (d\theta)^{2k} + d \Var^+ S_{\rho,k}.\label{eqn:high-snr-var-rec}
  \end{align}
  Solving \eqref{eqn:high-snr-var-rec} with initial value $S_{\rho,0} = 4\eta(1-\eta)$,
  we get
  \begin{align}
    \Var^+ S_{\rho,k} &= 4\eta(1-\eta)d^k + \sum_{i\in [k]} d^{k-i} d (1-2\eta)^2 (d\theta)^{2i-2} \nonumber\\
    & = 4\eta(1-\eta)d^k + (1-2\eta)^2 d^k \frac{(d\theta^2)^k-1}{d\theta^2-1}.\label{eqn:high-snr-var-result}
  \end{align}
  Putting together \eqref{eqn:high-snr-expectation}\eqref{eqn:high-snr-var-result}, we get the desired result.
\end{proof}

\section{Uniqueness of BP fixed point}\label{app:uniqueness}

  

\begin{prop}\label{prop:uniqueness-bp-fixed-point}
  Fix $d$, $\delta$, and a (possibly trivial) BMS $W$. Recall definition~\eqref{eqn:bp-fixed-point} of the BP fixed point (the $P_\Delta$ definition) for BOTS $(d, \theta, W)$. 
  \begin{itemize}
    \item If $W$ is non-trivial ($P_e(W) < \frac 12$) and $C_1 < 1$ (where $C_1$ is defined in \eqref{eqn:high-snr-defn-C1}), there is exactly one BP fixed point.
    \item If $W$ is trivial and $d\theta^2 \le 1$, there is exactly one BP fixed point, which is trivial (the point distribution at $\Delta=\frac 12$).
    \item If $W$ is trivial and $C_1 < 1$, there are exactly two BP fixed points, one is trivial and the other is non-trivial.
  \end{itemize}

  The same (statements about number of fixed points) hold for BOTS $(\Po(d), \theta, W)$
  with $C_1$ replaced by $C_2$ (defined in \eqref{eqn:high-snr-defn-C2}).
\end{prop}
\begin{proof}
If $W$ is trivial and $d\theta^2 \le 1$, we are in the non-reconstruction regime and there is a unique BP fixed point, and it is trivial.

If $W$ is trivial, there is one trivial fixed point.
If $W$ is non-trivial, the trivial distribution is not a fixed point.
We prove that for any $(d, \theta, W)$ satisfying $C_1<1$ (or $C_2<1$ for Poisson trees), there is exactly one non-trivial fixed point.

Suppose there are two non-trivial fixed points $P_\Delta$ and $Q_\Delta$.
Let $P$ be a BMS corresponding to $P_\Delta$ and $Q$ be a BMS corresponding to $Q_\Delta$.
Let $r = \max\{P_e(P), P_e(Q)\}$. Then $\BSC_r$ is non-trivial and is more degraded than both $P$ and $Q$.

We consider a Broadcasting on Tree model with three different types of observations:
\begin{itemize}
\item $M^a_k$: Observe $P(\sigma_v)$ for all $v\in L_k$;
\item $M^b_k$: Observe $Q(\sigma_v)$ for all $v\in L_k$;
\item $M^c_k$: Observe $\BSC_r(\sigma_v)$ for all $v\in L_k$.
\end{itemize}

By the same proof as Theorem \ref{thm:uniqueness-survey},
in the limit $k\to \infty$, $M^a_k$ and $M^c_k$ converge to the same BMS;
the same holds for $M^b_k$ and $M^c_k$.
Therefore in the limit $k\to \infty$, $M^a_k$ and $M^b_k$ converge to the same BMS.

By the assumption that $P$ and $Q$ are BP fixed points,
$M^a_k$ are equivalent to $P$ for all $k$,
and $M^b_k$ are equivalent to $Q$ for all $k$.
So $P$ and $Q$ are equivalent BMSs.
This means $P_\Delta = Q_\Delta$.
\end{proof}

\section{Rough estimate of $C$ in \cite{Mossel_2016}}\label{apx:mns_const}

As we mentioned, ~\cite{Mossel_2016} proves uniqueness of BP fixed point for BOT (without survey) and $d\theta^2 > C$ for an unspecified $C$. Can we extract explicit $C$ from their work? First, we point out that taken literally, the proof demands at least $C>75$. Second, we (heuristically!) argue below that it may be difficult to reduce $C$ below 25 without significant modifications of the proof. We remark that this section is not meant to be rigorous and it may very well be that the method therein can be tweaked in ways we did not consider. 

The proof in question is divided into the large $\theta$ case and the small $\theta$ case.
First, they prove that there exists a $\theta^*>0$ so that for $\theta \le \theta^*$, uniqueness of BP fixed point holds for large enough $d \theta^2$.
Then they prove that for $\theta > \theta^*$, there exists $d$ large enough so that uniqueness of BP holds.
We focus on the small $\theta$ part and analyze their proof for $\theta$ close to $0$.

In \cite[middle of page 2230]{Mossel_2016} authors require  $d\theta^2$ larger than about $75$. Let us analyze how much improvement is possible. In the following, equation and lemmas refer to the cited paper.
\begin{itemize}
\item In Lemma 3.6, it is impossible to achieve an RHS better than $1-\frac{1-\theta^2}{d\theta^2}$ by using a majority estimator (which is used by both their paper and the current paper).
\item In (3.8), they applied Lemma 3.9 with $p=\frac 14$. Changing this exponent would result in a big change in the proof, so we leave it as-is.
\item In Lemma 3.10, by Taylor expansion, it is impossible to improve RHS to $d m^{d-1} (\bE A^2 - \bE B^2)^2$.
\item In Lemma 3.11, by Taylor expansion $$\sqrt{\frac{1-x}{1+x}} = 1-x+\frac{x^2}2-\frac{x^3}2+O(x^4),$$ their proof cannot give a RHS better than
$1 - \theta^2 x_k + \frac 12 \theta^2$.
Combined with Lemma 3.6, their proof does not give a RHS better than 
$1 - \theta^2 (1-\frac{1-\theta^2}{d\theta^2}) + \frac 12 \theta^2$.
\item In Lemma 3.12, RHS cannot be better than $2\theta^2$, because this is less than $ |\frac{d}{dx} \frac{1-\theta^2x}{\sqrt{1-\theta^2x^2}}|$ at $x=-1$.
\item Consequently, in Lemma 3.13, the leading factor in RHS cannot be better than $2\theta^2$.
\item In (3.12), RHS cannot be better than $64d^2m^{d-2}(a-b)^2$ by using (3.8) with $p=\frac 14$ and Lemma 3.10.
\item Combining the above, in the expression in the middle of Page 2230, RHS cannot be better than
$$64((2\theta^2)^2 d^2 (1 - \theta^2 (1-\frac{1-\theta^2}{d\theta^2}) + \frac 12 \theta^2)^{d-2})z.$$
Computation shows that, for the factor before $z$ to be smaller than $1$, we need at least $d\theta^2 \ge 26$ in the limit $\theta \to 0$.
\end{itemize}

\section{Weak spatial mixing}\label{app:wsm}
In Section \ref{sec:recon-on-tree}, we studied whether BP message (with recursion \eqref{eqn:llr-recursion}) converges to the same value under perfect observation or no observation of leaves.
A related question is weak spatial mixing, i.e., whether BP message converges to the same value under any observation of leaves.

Fix $k\ge 0$.
Let $R_{L_k,0}$ and $R'_{L_k,0}$ be two boundary conditions.
Define $R_{\rho,k}$ (resp.~$R'_{\rho,k}$) by using \eqref{eqn:llr-recursion} recursively,
with initial condition $R_{L_k,0}$ (resp.~$R'_{L_k,0}$).
We say the model has weak spatial mixing if
\begin{align}
  \lim_{k\to \infty} \bE_{T, \omega_{T_k}} \sup_{R_{L_k,0}, R'_{L_k,0}} |f(R_{\rho,k})-  f(R'_{\rho,k})| = 0
\end{align}
for all bounded continuous functions $f : \mR\cup\{\pm \infty\} \to \mR$.

In the following we focus on regular trees. It is known \cite{bleher1995purity} that in the case there is no survey, $d\theta = 1$ is the threshold for WSM, i.e.,
when $d\theta < 1$, WSM holds; when $d\theta > 1$, WSM does not hold.
The following result shows that for WSM with survey, this is still almost the case.
\begin{thm}\label{thm:wsm}
\begin{itemize}
    \item For $d\theta < 1$ and any survey, WSM holds.
    \item For $d\theta > 1$, there exists $\epsilon = \epsilon(d, \theta)>0$ such that for $\BSC$ survey with $P_e(W) > \frac 12 -\epsilon$, WSM does not hold.
\end{itemize}
\end{thm}

\begin{proof}
  For $d\theta < 1$:
  For any node $u$, We have
  \begin{align}
    \bE |R_{u,k+1} - R'_{u,k+1}| &=  \bE |\sum_{v\in L_1(u)} (F_\theta(R_{v,k}) - F_\theta(R'_{v,k}))|\\
    &\le d\theta \bE |R_{v_1,k} - R'_{v_1,k}|.
  \end{align}
  (We use the fact that $F_\theta$ is $\theta$-Lipschitz.)
  So
  \begin{align}
    \bE |R_{\rho,k+1} - R'_{\rho,k+1}| \le d\theta \bE |R_{\rho,k} - R'_{\rho,k}|
  \end{align}
  and we get the desired contraction.
  
  For $d\theta > 1$:
  We separate the limit BP distribution for $(+)$-boundary condition and $(-)$-boundary condition.
  Because $d\theta > 1$, there exists $x > 0$ such that $d F_\theta(x) > x$.
  Let $\epsilon$ be small enough so that for all $\eta$ with $P_e(\BSC_\eta) > \frac 12-\epsilon$, we have
  \begin{align}
    d F_\theta(x) - \log \frac{1-\eta}{\eta} > x.
  \end{align}
  In this case, we can prove by induction that if we start with the $(+)$-boundary condition,
  then $R_{u,k} > x$ for all $u$ and $k$.
  By symmetry, if we start with the $(-)$-boundary condition, then $R_{u,k} < -x$ for all $u$ and $k$.
  So we get the desired separation.
\end{proof}

Note that for the case $d\theta>1$ we only prove for $\BSC$ survey.
Numerical computation suggests that the result should hold for any BMS survey with sufficiently large $P_e$. Thus we make the following conjecture.
\begin{conj}
  For $d\theta > 1$, there exists $\epsilon = \epsilon(d, \theta)>0$ such that for any BMS survey $W$ with $P_e(W) > \frac 12-\epsilon$, WSM does not hold.
\end{conj}

\section{Amenable graphs}  \label{app:amenable}
Recall definition of the spin synchronization system and the (BI) property given in~\eqref{eq:BI-SOG}.

\begin{defin}[Amenable graph (\cite{alaoui2019computational})]
A graph $G$ is said to be amenable if $\inf \{|\partial S| /|S|: S \subset V \text { finite, } o \in S\}=0$, where $ \partial S=\{u \in S: \exists v \notin S,(u, v) \in E\}$.
\end{defin}
\begin{thm} \label{thm:amenable}
Let $G$ be an amenable graph. For any $\epsilon \in [0,1)$, the (BI) holds for $(G,o,\theta,\BEC_\epsilon)$.
\end{thm}

\begin{proof} 
A consistent part of this proof is inspired by Lemma 6.3 in \cite{alaoui2019computational}. We reproduce it for a self-contained exposure. As in the proof of Theorem \ref{thm:HSBM-BOT}, let us replace the single parameter $\epsilon$ by a set of parameters $(\epsilon_u)_{u \in V(G)}$ (for each vertex $u$, $X_u$ is revealed with probability $1-\epsilon_u$), and let us denote $X_{\sim u}^\epsilon = \{ X_v^\epsilon: v \in V(G), v\neq u \}$. For brevity, we write $B_n, \partial B_n $ for $B_n(o), \partial B_n(o) $. Then,
\begin{align*}
    \frac{\partial }{\partial \epsilon_u} H( X_{\partial B_n} \mid Y,X^\epsilon) = I(X_u; X_{\partial B_n} \mid Y,X_{\sim u}^\epsilon),
\end{align*}
and setting $\epsilon_u=\epsilon$ for every $u \in B_n$ we get 
\begin{align}
    \frac{d}{d\epsilon }H( X_{\partial B_n} \mid Y,X^\epsilon) = \sum_{u \in B_n} I (X_u; X_{\partial B_n} \mid Y,X_{\sim u}^\epsilon).
\end{align} 
Thus, integrating with respect to $\epsilon$ we get
\begin{align}
    \int_{\epsilon}^1 \sum_{u \in B_n} I (X_u; X_{\partial B_n} \mid Y,X_{\sim u}^{\epsilon'}) d\epsilon '& = H(X_{\partial B_n} \mid Y) -  H(X_{\partial B_n} \mid Y , X^{\epsilon})\\
    & \leq H( X_{\partial B_n} ) \\
    & \leq \sum_{u \in \partial B_n}  H(X_u) \\
    & = \log{2}|\partial B_n|.
\end{align}
If we divide by $|B_n|$, we get that for all $\epsilon<1$
\begin{align}
\int_{\epsilon}^{1} \frac{1}{|B_n|} \sum_{u \in B_n}  I\left(X_{u} ; X_{\partial B_n} \mid Y_{B_n}, X_{B_n}^{\epsilon'} \right) \mathrm{d} \epsilon^{\prime} \leq \log{2} \cdot \frac{|\partial B_n|}{|B_n|}.
\end{align}
Since $G$ is amenable, the RHS is vanishing as $n \to \infty$. Note that the integrand in the LHS is bounded by $\log{2}$, hence by bounded convergence theorem, we get that for all $\epsilon \in [0,1)$
\begin{align}
\lim_{n \to \infty} \frac{1}{|B_n|} \sum_{u \in B_n}  I\left(X_{u} ; X_{\partial B_n} \mid Y_{B_n}, X_{B_n}^{\epsilon} \right)  = 0.
\end{align}
To conclude, notice that there exists $k \in \mathbb{N}$ such that 
\begin{align}
I(X_o; X_{\partial B_{k \cdot n}(o)} | Y, X^\epsilon) 
&\leq \frac{1}{|B_n(o)|} \sum_{u \in B_n(o)}  I (X_{u} ; X_{\partial B_n(o)} \mid Y, X^{\epsilon} ).
\end{align}
\end{proof}

\end{document}